\documentclass[10pt]{amsart}

\usepackage[latin1]{inputenc}
\usepackage{amsmath}
\usepackage{amsfonts}
\usepackage{amssymb}
\usepackage[mathscr]{eucal}
\usepackage{diagrams}
\usepackage{xy,color,graphicx}
\usepackage{hyperref}
\xyoption{all}

\newcommand{\wis}[1]{\mathbf{#1}}

\newcommand{\PP}{\mathbb{P}}
\newcommand{\C}{\mathbb{C}}

\newtheorem{theorem}{Theorem}
\newtheorem{lemma}{Lemma}

\title{The superpotential  $\mathbf{XYZ+XZY + \frac{c}{3} (X^3+Y^3+Z^3)}$}
\author{Lieven Le Bruyn} 
\address{Department of Mathematics, University of Antwerp \\ 
 Middelheimlaan 1, B-2020 Antwerp (Belgium) \\ {\tt lieven.lebruyn@uantwerpen.be}}

\begin{document}
\sloppy

\maketitle

\begin{abstract}
The motivic Donaldson-Thomas series associated to an elliptic Sklyanin algebra corresponding to a point of order two differs from the conjectured series in \cite[Conjecture 3.4]{Cazz}.
\end{abstract}

\section{Introduction}

A {\em $3$-dimensional elliptic Sklyanin algebra} $S=S_{a,b,c}$ is a quotient of the free algebra $\C \langle X,Y,Z \rangle$ modulo the graded ideal generated by the three quadratic relations
\[
\begin{cases}
a XY + b YX + c Z^2 &= 0 \\
a YZ + b ZY + c X^2 &= 0 \\
a ZX + b XZ + c Y^2 &= 0 
\end{cases}
\]
If $abc \not= 0$ and $3(abc)^3 \not= (a^3+b^3+c^3)^3$ these algebras have excellent ringtheoretic and homological properties, as proved by M. Artin, J. Tate and M. Van den Bergh in \cite{ATV1},\cite{ATV2}. They are determined by the plane elliptic curve 
\[
E_{pt}~:~(a^3+b^3+c^3)XYZ - abc (X^3+Y^3+Z^3) = 0 \subset \mathbb{P}^2 \]
 and translation by the point $\tau=[a:b:c] \in E_{pt}$ on it. The tools of noncommutative projective algebraic geometry have been used to classify the finite dimensional simple representations of $S_{a,b,c}$ in case $\tau \in E_{pt}$ is a point of finite order, see \cite{Walton}, \cite{Kevin}, and more recently \cite{Walton2}. We recall these result in section~\ref{simples} and make them explicit in the case when $\tau$ has order two, using the theory of Clifford algebras.
 
 \vskip 2mm
 
 The Sklyanin algebra $S_{a,b,c}$ can also be realized as the Jacobi algebra associated to the superpotential
 \[
 W = a XYZ + b XZY + \frac{c}{3}(X^3+Y^3+Z^3) \]
 That is, if $\partial_V$ denotes the cyclic derivative with respect to the variable $V$, then
 \[
 S_{a,b,c} = \frac{\C \langle X,Y,Z \rangle}{(\partial_X(W),\partial_Y(W),\partial_Z(W))} \]
$Tr(W)$ determines the Chern-Simons functional $M_n(\C) \oplus M_n(C) \oplus M_n(\C) \rTo \C$ and for every $\lambda \in \C$ we will denote by $\mathbb{M}^W_n(\lambda)$ the fiber $Tr(W)^{-1}(\lambda)$. Because the degeneracy locus of $Tr(W)$ coincides with the scheme of $n$-dimensional representations of $S_{a,b,c}$ it is conjectured in \cite{Cazz} that the motivic Donaldson-Thomas series
\[
U_W(t) = \sum_{n=0}^\infty \mathbf{L}^{\frac{-2n^2}{2}} \frac{[\mathbb{M}^W_n(0)]-[\mathbb{M}^W_n(1)]}{[GL_n]} t^n \]
is determined by the virtual motives of simple representations of $S_{a,b,c}$. If $\tau$ has order $n$ and $(n,3)=1$ it is known that apart from the trivial $1$-dimensional representation all finite dimensional simple representations of $S_{a,b,c}$ have dimension $n$ and \cite[Conjecture 3.4]{Cazz} conjectures that in this case we have
\[
U_W(t) =\mathbf{Exp}(-\frac{M_1}{\mathbb{L}^{\frac{1}{2}}-\mathbb{L}^{-\frac{1}{2}}} \frac{t}{1-t} - \frac{M_n}{\mathbb{L}^{\frac{1}{2}}-\mathbb{L}^{-\frac{1}{2}}} \frac{t^n}{1-t^n}) \]
with $M_1 = \mathbb{L}^{-\frac{3}{2}}([X_{DT}=1,\mu_3]-[X_{DT}=0])$ where $X_{DT}$ is the cubic in $\mathbb{A}^3$
\[
X_{DT} = (a+b)xyz + \frac{c}{3}(x^3+y^3+z^3) \]
and where $M_n = \mathbb{L}^{1/2}([\mathbb{P}^2]-[E_c])$ where $E_c$ is the plane elliptic curve $E_{pt}/ \langle \tau \rangle$ isogenous to $E_{pt}$ by dividing out the cyclic subgroup generated by $\tau$.

\vskip 2mm

In \cite{LBDT} we developed a method to verify such conjectures inductively by calculating the motives of certain Brauer-Severi schemes. In this paper we will compute the second term of $U_W(t)$ for the Sklyanin algebra $S_{1,1,c}$, that is when $\tau$ is a point of order two. By \cite[Conjecture 3.4]{Cazz} one would expect this coefficient to involve the motives of at least two different elliptic curves $[E_c]$ and $[E_{DT}]$ (which have different $j$-invariants). However, the computed term only involves the motif $[E_{DT}]$.

\vskip 4mm

\noindent
Acknowledgement : I like to thank Balazs Szendr\"oi for generous help in calculating the coefficient of the second term in the conjectured plethystic exponential, see section~\ref{conjecture}, and Brent Pym and Ben Davison for their continuing interest in this project.

\section{Simple representations of Sklyanin algebras} \label{simples}

The elliptic curve associated to the Sklyanin algebra $S_{a,b,c}$
\[
E_{pt}~:~(a^3+b^3+c^3)XYZ - abc (X^3+Y^3+Z^3) = 0 \]
is the locus of all {\em point modules} of $S_{a,b,c}$, that is, graded (critical) left-modules $A/(Al_1+Al_2)$ with the $l_i$ linear in $X,Y,Z$ (and hence $l_1,l_2$ determine a point in $\mathbb{P}^2$) such that its Hilbert series is $(1-t)^{-1}$. Addition by the point $p = [a:b:c] \in E_{pt}$ describes the automorphism on point modules given by the shift-by-$1$ functor. A {\em line module} of $S_{a,b,c}$ is a graded (critical) left-module $A/Al$ with $l$ linear and Hilbert series $(1-t)^{-2}$. As $S_{a,b,c}$ is a domain, line modules correspond to lines in $\mathbb{P}^2$.

We are particularly interested in elliptic Sklyanin algebras which are finite modules over their centers. S. P. Smith and J. Tate  \cite{SmithTate} proved that this is the case if and only if $\tau \in E_{pt}$ is a point of finite order $n$. In this case $S_{a,b,c}$ is a maximal order in a division algebra of dimension $n^2$ over its center and the center of $S_{a,b,c}$ is isomorphic to
\[
Z_{a,b,c} = \frac{\C[u_1,u_2,u_3,c_3]}{\Phi(u_1,u_2,u_3)-c_3^n} \]
where the $u_i$ are central elements of degree $n$, $c_3$ is a central element of degree $3$ and $\Phi$ is a homogeneous polynomial of degree $3$ in the $u_i$ describing the isogenous elliptic curve $E_c = E_{pt}/\langle \tau \rangle$. In \cite{Walton} and \cite{Kevin} it is shown that when $(n,3)=1$ all finite dimensional simple representations of $S_{a,b,c}$ (apart from the trivial $1$-dimensional simple) are of dimension $n$ and correspond to the smooth points of the central variety, which has an isolated singularity at the top. 

In principle, one can give an explicit description of the triple of $n \times n$ matrices describing the simple $n$-dimensional representation $M_q$ corresponding to the maximal  (non-graded) ideal $\mathfrak{m}_q$ of $Z_{a,b,c}$ using the isogeny $E_{pt} \rOnto E_c$, see \cite{LBsymbol} or \cite{Kevin}. If $c_3$ does not vanish in $q$, the ruling from the top-singularity through $q$ determines a point $\overline{q}$ in $\wis{Proj}(Z_{a,b,c}) = \mathbb{P}^2 = \mathbb{P}(u_1^*,u_2^*,u_3^*)$ not lying on the elliptic curve $E_c$. Write $\overline{q}$ as the intersection of two lines $L_1$ and $L_2$ in $\mathbb{P}^2$ and lift $L_1$ through the isogeny to a line $L$ in $\mathbb{P}^2=\mathbb{P}(X^*,Y^*,Z^*)$, then $\overline{q}$ determines the {\em fat point of multiplicity $n$}, that is, the graded (critical) left-module with Hilbert series $n/(1-t)$
\[
F_{\overline{q}} = \frac{A}{Al+Al_2} \]
where $l$ is the linear form in $X,Y,Z$ determining $L$ and $l_2$ the degree $n$ central element which is the linear form in $u_1,u_2,u_3$ determining $L_2$. The central localization of $S_{a,b,c}$ at $c_3$ has a central element $t$ of degree $1$ and the simple representation $M_q$ is then the quotient of $F_{\overline{q}}$ by $t-\lambda$ where $\lambda$ is the evaluation of $t$ in $q$. If $c_3$ is zero in $q$, the ruling determines a point $\overline{q} \in E_c$ which lifts through the isogeny to $n$ point modules which form of $\tau$-orbit. The coordinates of the corresponding $n$ points on $E_{pt}$ can then be used to give explicit $n \times n$ matrices of the corresponding simple representation $M_q$, see \cite[\S 3.1]{Kevin}.

Clearly, this approach is only as effective as we have explicit formulas for lifting through the isogeny $E_{pt} \rOnto E'$, that is for small $n$. Next, we give explicit matrices describing the simple representations in the case when $n=2$, that is when $a=b=1$, not using the isogeny but the fact that in this case the Sklyanin algebras $S_c = S_{1,1,-c}$ can be viewed as Clifford algebras of ternary symmetric bilinear forms and we can apply the theory of quadratic forms to describe its simple $2$-dimensional representations.

\vskip 4mm

In a recent paper \cite{ReichWalton} D.J. Reich and C. Walton describe a Maple algorithm to obtain explicit representations of $3$-dimensional Sklyanin algebras associated to a point of order two. Here we give a pen-and-paper approach, using  classical quadratic form theory.

Let $A=(a_{ij})_{i,j} \in M_3(\C)$ be a symmetric $3 \times 3$ matrix of rank $\geq 2$. The associated {\em Clifford algebra} $\wis{Cliff}_{\C}(A)$ is the $8$-dimensional $\C$-algebra generated by three elements $x_1,x_2$ and $x_3$ with defining relations
\[
x_i.x_j + x_j.x_i = a_{ij} \qquad \text{for all $1 \leq i,j \leq 3$} \]
The symmetric bilinear form on $V= \C x_1 + \C x_2 + \C x_3$ defined by $A$ coincides with $\langle v , w \rangle = Tr(v.w)$ for all $v,w \in V$, where the product is taken in the Clifford algebra. The structure of Clifford algebras is well-known, see for example \cite{Lam}.
\[
\wis{Cliff}_{\C}(A) \simeq \begin{cases} M_2(\C) \oplus M_2(\C)~& \text{if $rk(A)=3$} \\
M_2(\C) \otimes \C[\epsilon]~&\text{if $rk(A)=2$} \end{cases}
\]
That is, $\wis{Cliff}_{\C}(A)$ has two distinct simple $2$-dimensional representations $\psi_{\pm}$, which coincide when $det(A)=0$. We want to describe these explicitly, that is determine the $2 \times 2$ matrices $\psi_{\pm}(x_i)$. There is an invertible matrix $P \in GL_3(\C)$ such that
\[
P^{\tau}.A.P = \begin{bmatrix} 1 & 0 & 0 \\ 0 & 1 & 0 \\ 0 & 0 & 1 \end{bmatrix} = \langle 1,1,1 \rangle \quad \text{or} \quad \begin{bmatrix} 1 & 0 & 0 \\ 0 & 1 & 0 \\ 0 & 0 & 0 \end{bmatrix} = \langle 1,1,0 \rangle \]
The {\em Pauli matrices} describe the simple representations of $\wis{Cliff}_{\C}(\langle 1,1,\delta \rangle)$. If
\[
\sigma_1 = \begin{bmatrix} 0 & 1 \\ 1 & 0 \end{bmatrix}, \quad \sigma_2 = \begin{bmatrix} 0 & i \\ -i & 0 \end{bmatrix} \quad \text{and} \quad \sigma_3 = \begin{bmatrix} 1 & 0 \\ 0 & -1 \end{bmatrix} \]
then we have
\[
\psi_{\pm}(u_1) = \sigma_1, \quad \psi_{\pm}(u_2) = \sigma_2 \quad \text{and} \quad \psi_{\pm}(u_3) = \pm \delta \sigma_3 \]
for the new basis $(u_1,u_2,u_3)^{\tau} = P.(x_1,x_2,x_3)^{\tau}$ of $V$. But then, if $P^{-1}=(q_{ij})_{i,j}$ we have:

\begin{lemma} The simple $2$-dimensional representation(s) of $\wis{Cliff}_{\C}(A)$ are given by
\[
\psi_{\pm}(x_i) = \sum_{j=1}^3 q_{ji} \psi_{\pm}(u_j) = q_{i1} \sigma_1 + q_{i2} \sigma_2 + \pm q_{i3} \delta \sigma_3 \]
\end{lemma}

The {\em $3$-dimensional quaternion Sklyanin algebra} $S_c=S_{1,1,-c}$ is the $\C$-algebra generated by three elements $X=x_1,Y=x_2,Z=x_3$ with defining quadratic relations
\[
XY+YX=cZ^2, \quad YZ+ZY=cX^2 \quad \text{and} \quad ZX+XZ=Y^2 \]
It follows that $u=X^2,v=Y^2$ and $Z^2=w$ are central elements and hence that $S_c$ is the Clifford algebra over $R=\C[u,v,w]$ as in \cite{Bass} associated with the ternary symmetric bilinear form on the free module $V= R x_1 \oplus R x_2 \oplus R x_3$ determined by the symmetric matrix in $M_3(R)$
\[
Q = \begin{bmatrix} 2 u & c w & c v \\
c w & 2 v & cu \\
cv & cu & 2 w \end{bmatrix} \]
Evaluating the entries of $Q$ in a point $p = (\alpha,\beta,\gamma) \in \mathbb{A}^3_{\C} = \wis{max}(R)$ we obtain a symmetric matrix $A=Q(p) \in M_3(\C)$ which is of rank at least two if and only if $p \not= (0,0,0)$. Lemma~1 gives us explicit representations of the two (or one) simple $2$-dimensional representations $\psi_{\pm}(p)$ of $S_c$ lying over the point $p$. 

It follows from \cite{LBVdB} or \cite{SmithTate} that the center $Z(S_c) = R \oplus R.Tr(x_1x_2x_3)$ where $Tr(x_1x_2x_3)^2 = D = det(Q)$. As a result $\wis{max}(Z(S_c))$ is a two-fold cover of $\mathbb{A}^3_{\C} = \wis{max}(R)$ ramified along the surface where $D$ vanishes. By the above, points of $\wis{max}(Z(S_c))$ (apart from the unique point lying over $0=(0,0,0)$) are in one-to-one correspondence with the isomorphism classes of $2$-dimensional simple representations of $S_c$.

We will now construct families of explicit representations as in \cite{ReichWalton}. The idea is to diagonalize $Q$ over $\mathbb{A}^3-\{ 0 \}$ and to keep track of the base-change matrix $P \in M_3(\C[u,v,w])$. For this we apply the classical diagonalization algorithm which in this case involves the choice of just two pivots.

As $p \not= (0,0,0)$ we may assume (after permuting the variables $x_i$ if necessary) that $2u \not= 0$ which will be our first pivot. One starts off with the $3 \times 6$ matrix $(Q | I_3)$ and uses the pivot to obtain zeroes in positions $2,3$ of the first column and positions $2,3$ in the first row by the usual trick of adding suitable multiples of rows and columns. The row-operations also have an effect on the right-hand side $3 \times 3$ matrix. After this step one obtains the matrix
\[
\begin{bmatrix} 2 u & 0 & 0 & 1 & 0 & 0 \\
0 & 2 u (4 uv -c^2 w^2) & 2u (2 c u^2 - c^2 vw) & - c w & 2 u &  0 \\
0 & 2u (2c u^2-c^2 v w) & 2u (4 u w - c^2 v^2) & - c v & 0 & 2u \end{bmatrix} \]

{\bf Case 1 :} If $A=4 uv - c^2w^2 \not= 0$ (or, after permuting the variables, $4uw-c^2 v^2 \not= 0$) use this as pivot. After this step one obtains the diagonal matrix $\Delta$ and the base-change matrix $P$
\[
(\Delta | P^{\tau}) = \begin{bmatrix}
2u & 0 & 0 & 1 & 0 & 0 \\
0 & 2u A & 0 & -c w & 2u & 0 \\
0 & 0 & 4 u^2 A D & 2 c u B & 2 c u C & 2 u A \end{bmatrix} \]
where $B = c u w - 2 v^2$ and $C = c v w - 2 u^2$. Clearly, $P$ is invertible on the open set where $u A \not= 0$.

{\bf Case 2 :} If $4 u v - c^2 w^2 = 0 = 4 uw -c^2 v^2$, we have $2 cu^2 -c^2 vw \not= 0$. In this case we add the third row to the second and the third column to the second, use the resulting $(2,2)$-entry as pivot in order to arrive at
\[
(\Delta | P^{\tau}) = \begin{bmatrix}
2 u & 0 & 0 & 1 & 0 & 0 \\
0 & -2 u L & 0 & -c v - c w & 2 u & 2 u \\
0 & 0 & -16 u^4 L D & 4 c u^2 Q_0 & 4 u^2 Q_1 & -4 u^2 Q_2
\end{bmatrix}
\]
where
\[
\begin{cases}
Q_0 &= (w-v)(2w + 2v + cu) \\
Q_1 &= c^2 vw - 4 uw +c^2v^2 - 2c u^2) \\
Q_2 &= c^2 w^2+c^2 v w - 4 u v - 2 c u^2)
\end{cases}
\]
and $L = Q_1 + Q_2$. The determinant of the basechange matrix is $-8 u^3 L$. In a point where $4 uv - c^2 w^2=0=4 u w - c^2 v^2$, $L$ is equal to $-2(2cu^2-c^2 vw)$ so $P$ is invertible in those points. Observe that these two cases cover all points in $\wis{max}(Z(S_c))$ where $u \not= 0$.

\begin{lemma} With notations as above, let $\Delta = diag(D_1,D_2,D_3)$ and $P^{-1} = (Q_{ij})_{i,j}$. Then, the maps (remember that $x_1=X,x_2=Y$ and $x_3=Z$)
\[
\psi_{\pm}(x_i) = Q_{i1} \sqrt{D_1} \sigma_1 + Q_{i2} \sqrt{D_2} \sigma_2 \pm \sqrt{D_3} \sigma_3 \]
give a family of explicit representations of $S_c$, with a unique representative for all simple $2$-dimensional representations on the open set of $\wis{max}(Z(S_c))$ where $u \not= 0$. Here we take the matrices of the first case if $uA \not= 0$ and those of the second case on the locus where $4 u v - c^2 w^2 = 0 = 4 uw -c^2 v^2$. Permuting the variables covers the entire Azumaya-locus of $S_c$ which is $\wis{max}(Z(S_c))$ with the unique isolated singularity lying over $(0,0,0)$ removed.
\end{lemma}

For example, on the open set where $u A \not= 0$ we have the following explicit matrix-representations:
\[
\begin{cases}
\psi_{\pm}(X) &= \begin{bmatrix} 0 & \sqrt{2u} \\ \sqrt{2u} & 0 \end{bmatrix} \\
& \\
\psi_{\pm}(Y) &= \frac{c w}{2u} \begin{bmatrix} 0 & \sqrt{2u} \\ \sqrt{2u} & 0 \end{bmatrix}  + \frac{1}{2u} \begin{bmatrix} 0 & -i \sqrt{2u A} \\ i \sqrt{2 u A} & 0 \end{bmatrix} \\
& \\
\psi_{\pm}(Z) &= \frac{c v}{2u} \begin{bmatrix} 0 & \sqrt{2u} \\ \sqrt{2u} & 0 \end{bmatrix} - \frac{cC}{2 u A} \begin{bmatrix} 0 & -i \sqrt{2u A} \\ i \sqrt{2 u A} & 0 \end{bmatrix} \mp \frac{1}{2 u A} \begin{bmatrix}  2u \sqrt{AD} & 0 \\ 0 & - 2u \sqrt{AD} \end{bmatrix} 
\end{cases} \]

\section{Superpotentials and motives} \label{conjecture}

Consider the cubic superpotential $W=aXYZ+bXZY+\tfrac{c}{3}(X^3+Y^3+Z^3)$ in the noncommutative variables $X,Y$ and $Z$. For every dimension $n \geq 1$, the superpotential $W$ determines the Chern-Simons functional
\[
Tr(W)~:~M_n(\C) \oplus M_n(\C) \oplus M_n(\C) \rTo \C \]
obtained by replacing $X,Y$ and $Z$ by the first, second resp. third component matrix and taking the trace. The representation theoretic interest of the degeneracy locus $\{~d Tr(W)=0~\}$ of this functional is that it coincides with the scheme of $n$-dimensional representations $\wis{rep}_n(R_W)$ of the associated Jacobi algebra
\[
R_W = \frac{\C \langle X,Y,Z \rangle}{(\partial_X(W),\partial_Y(W),\partial_Z(W))} \]
where the $\partial_V$ are the cyclic derivative with respect to the variables $V$, which in the case of the above superpotential $W$
gives us the defining equations of $S_{a,b,c}$. That is, the degeneracy locus of the superpotential $W$
\[
\{~d Tr(W) = 0~\} = \wis{rep}_n(S_{a,b,c}) \]
By the Denef-Loeser theory of motivic nearby cycles, see \cite{Denef}, the motive of this degeneracy locus can often be computed as the difference of the motives of the general fiber and the zero-fiber of the functional. For this reason we are interested in the (naive, equivariant) motive of the $\lambda$-fiber of the functional $Tr(W)$ which we denote by $\mathbb{M}^W_n(\lambda) = Tr(W)^{-1}(\lambda)$.

Recall that to each isomorphism class of a complex variety $X$ (equipped with a good action of a finite group of roots of unity) we associate its naive equivariant  motive $[X]$ which is an element in the ring $K_0^{\hat{\mu}}(\mathrm{Var}_{\C})[ \mathbb{L}^{-1/2}]$ (see \cite{Davison} or \cite{Cazz}) and is subject to the scissor- and product-relations
\[
[X]-[Z]=[X-Z] \quad \text{and} \quad [X].[Y]=[X \times Y] \]
whenever $Z$ is a Zariski closed subvariety of $X$. A special element is the Lefschetz motive $\mathbb{L}=[ \mathbb{A}^1_{\C}, id]$ and we recall from \cite[Lemma 4.1]{Morrison} that $[GL_n]=\prod_{k=0}^{n-1}(\mathbb{L}^n-\mathbb{L}^k)$ and from \cite[2.2]{Cazz} that $[\mathbb{A}^n,\mu_k]=\mathbb{L}^n$ for a linear action of $\mu_k$ on $\mathbb{A}^n$. This ring is equipped with a plethystic exponential $\wis{Exp}$, see for example \cite{Bryan} and \cite{Davison}.

As $W$ is homogeneous it follows from \cite[Thm. 1.3]{Davison}  that the virtual motive of the degeneracy locus is equal to
\[
[ dTr(W)=0 ]_{virt} = [ \wis{rep}_n(S_{a,b,c}) ]_{virt} = \mathbb{L}^{-\frac{2n^2}{2}}([\mathbb{M}_{n}^W(0)]-[\mathbb{M}_{n}^W(1)]) \]
where $\hat{\mu}$ acts via $\mu_d$ on $\mathbb{M}^W_{n}(1)$ and trivially on $\mathbb{M}^W_{n}(0)$.
These virtual motives can be packaged together into the motivic Donaldson-Thomas series
\[
U_W(t) = \sum_{n=0}^{\infty} \mathbb{L}^{- \frac{2n^2}{2}} \frac{[\mathbb{M}_{n}^W(0)]-[\mathbb{M}_{n}^W(1)]}{[GL_n]} t^n \]
By the Jordan-H\"older theorem, the sequence $ \{~[ \wis{rep}_n(S_{a,b,c}) ]_{virt}~\}$ is expected to jump at every dimension $n$ where $S_{a,b,c}$ has simple $n$-dimensional representations. For this reason A. Cazzaniga, A. Morrison, B. Pym and B. Szendr\"oi conjecture in \cite{Cazz} that the generating sequence $U_W(t)$ has an exponential expression involving rational functions of virtual motives connected to the simple representations of the Jacobi algebra $S_{a,b,c}$. Explicitly, their conjecture \cite[Conjecture 3.4]{Cazz} asserts that in case $\tau \in E_{pt}$ has infinite order that then
\[
U_W(t) = \wis{Exp}(- \frac{M_1}{\mathbb{L}^{\tfrac{1}{2}} - \mathbb{L}^{- \tfrac{1}{2}}}\frac{t}{1-t}) \]
where $M_1= \mathbb{L}^{-3/2}([X_{DT} = 1] - [X_{DT} = 0])$ where $X_{DT}$ is the cubic function in the three commuting variables $x,y,z$
\[
X_{DT} = (a+b)xyz + \frac{c}{3}(x^3+y^3+z^3) \]
which gives $Tr(W)$ for $n=1$. Note that $X_{DT}$ determines an elliptic curve in $\PP^2$, usually with a different $j$-invariant than $E_{pt}$ and $E_c$. If however $\tau \in E_{pt}$ is a point of finite order $n$ and $(n,3)=1$ one expects another term in the exponential expression coming from the simples in dimension $n$. In \cite[Conjecture 3.4]{Cazz} it is conjectured that in this case
\[
U_W(t) = \wis{Exp}(- \frac{M_1}{\mathbb{L}^{\tfrac{1}{2}} - \mathbb{L}^{- \tfrac{1}{2}}}\frac{t}{1-t} - \frac{M_n}{\mathbb{L}^{\tfrac{1}{2}}-\mathbb{L}^{-\tfrac{1}{2}}}\frac{t^n}{1-t^n}) \]
where $M_n=\mathbb{L}^{1/2}([\mathbb{P}^2] - [E_c])$. Observe already from section~\ref{simples} that this term only encodes the simple $n$-dimensional representations determined by points $q \in \wis{Spec}(Z_{a,b,c})$ not lying on the cone over $E_c$.

\begin{lemma} If we denote with
\[
N_1 = (\mathbb{L}-1)[E_{DT}]+1-[S_{DT},\mu_3] \quad \text{and} \quad N_2=[E_c]-[\mathbb{P}^2] \]
then the coefficient of $t^2$ in the conjectured series $U_W(t)$ is equal to
\[
\frac{\mathbb{L}(\mathbb{L}^2-1) N_2 + \mathbb{L}^{-2} N_1^2 + \mathbb{L}^{-1}(\mathbb{L}^2-1) N_1 + \mathbb{L}^{-2} (\mathbb{L}-1) \sigma_2(N_1)}{(\mathbb{L}^2-1)(\mathbb{L}-1)} \]
\end{lemma}

\begin{proof}
With these notations, the conjecture \cite[Conjecture 3.4]{Cazz} can be rewritten as
\[
U_W(t) = \wis{Exp}(\frac{\mathbb{L}(\mathbb{L}^{-2} N_1)}{\mathbb{L}-1}\frac{t}{1-t}).\wis{Exp}(\frac{\mathbb{L}N_2}{\mathbb{L}-1}\frac{t^2}{1-t^2}) \]
The second term is equal to
\[
\wis{Exp}(\sum_{k\geq 1} \sum_{j \geq 0} \mathbb{L}^{-j} N_2 t^{2k}) = \prod_{k \geq 1}\prod_{j \geq 0} \wis{Exp}(\mathbb{L}^{-j} N_2 t^{2k}) = \]
\[
\prod_{k \geq 1} \prod_{j \geq 0}(\sum_{n \geq 0} \sigma_n(\mathbb{L}^{-j} N_2 t^{2k})) = \prod_{k \geq 1} \prod_{j \geq 0} (\sum_{n \geq 0} \mathbb{L}^{-nj} \sigma_n(N_2) t^{2kn}) \]
As we are only interested in the coefficient of $t^2$ we need only consider the term in the first product where $k=1$ and then get
\[
(1 + N_2 t^2 + \hdots)(1 + \mathbb{L}^{-1} N_2 t^2+ \hdots)(1 + \mathbb{L}^{-2} N_2 t^2 + \hdots) \hdots = 1 +  \frac{N_2}{1- \mathbb{L}^{-1}} t^2 + \hdots \]
For the first term, we get likewise
\[
\wis{Exp}(\sum_{k \geq 1} \sum_{j \geq 2} \mathbb{L}^{-j} N_1 t^k) = \prod_{k \geq 1} \prod_{j \geq 2} \wis{Exp}(\mathbb{L}^{-j} N_1 t^k) = \]
\[
\prod_{k \geq 1} \prod_{j \geq 2} (\sum_{n \geq 0} \sigma_n(\mathbb{L}^{-j} N_1 t^k)) = \prod_{k \geq 1} \prod_{j \geq 2} (\sum_{n \geq 0} \mathbb{L}^{-nj} \sigma_n(N_1) t^{kn}) \]
As we only want the coefficient of $t^2$ we have to consider three contributions:

$k=1,n=1$ in two brackets with $j_2 > j_1 \geq 2$ this gives
\[
\sum_{2 \leq j_1 < j_2} N_1^2 \mathbb{L}^{-(j_1+j_2)} = \sum_{j \geq 2} \sum_{k \geq 0} \mathbb{L}^{-2j-k-1} N_1^2 = \]
\[
 \mathbb{L}^{-5} N_1^2 (\sum_{j \geq 0} \mathbb{L}^{-2j})(\sum_{k \geq 0} \mathbb{L}^{-k})
= \frac{\mathbb{L}^{-5} N_1^2}{(1 - \mathbb{L}^{-2})(1 - \mathbb{L}^{-1})} = \frac{\mathbb{L}^{-2} N_1^2}{(\mathbb{L}^2-1)(\mathbb{L}-1)} \]

$k=2, n=1$ in one bracket and $n=0$ in all others. This gives
\[
\sum_{j \geq 2} \mathbb{L}^{-j} N_1 = \frac{\mathbb{L}^{-2} N_1}{1 - \mathbb{L}^{-1}} = \frac{\mathbb{L}^{-1} N_1}{\mathbb{L}-1} \]

$k=1, n=2$ in one bracket and $n=0$ in all others. Then we get
\[
\sum_{j \geq 2} \mathbb{L}^{-2j} \sigma_2(N_1) = \frac{\mathbb{L}^{-4} \sigma_2(N_1)}{1 - \mathbb{L}^{-2}} = \frac{\mathbb{L}^{-2} \sigma_2(N_1)}{\mathbb{L}^2 -1} \]
Summing up all terms gives the claimed expression.
\end{proof}

\section{Brauer-Severi motives}

In \cite{LBDT} an inductive method was proposed to compute the coefficients of the series $U_W(t)$ inductively. For every $n \geq 1$ and every $\lambda \in \C$ introduce the following quotient of the trace ring $\mathbb{T}_{3,n}$ of $3$ generic $n \times n$ matrices
\[
\mathbb{T}^W_n(\lambda) = \frac{\mathbb{T}_{3,n}}{(Tr(W)-\lambda)} \]
The reason being that the $\lambda$-fiber $Tr(W)^{-1}(\lambda)$ is the scheme of $n$-dimensional trace preserving representations of $\mathbb{T}^W_n(\lambda)$
\[
Tr(W)^{-1}(\lambda) = \wis{trep}_n(\mathbb{T}^W_n(\lambda)) \]
Now, consider the associated Brauer-Severi scheme in the sense of M. Van den Bergh \cite{VdBBS}. That is, consider the open subscheme $U^W_n$ of $\wis{trep}_n(\mathbb{T}^W_n(\lambda)) \times \C^n$ consisting of couples
\[
U^W_n(\lambda) = \{ (\phi,v) \in \wis{trep}_n(\mathbb{T}^W_n(\lambda)) \times \C^n~|~\phi(\mathbb{T}^W_n(\lambda)).v = \C^n \} \]
on which $GL_n$ acts freely and let the Brauer-Severi scheme be the corresponding quotient variety $\wis{BS}^W_n(\lambda) = U^W_n(\lambda)/GL_n$. Then it is shown in \cite[Prop. 5]{LBDT} that one can compute the fiber-motives at $n$ from knowledge of the Brauer-Severi-motives for all dimensions $k \leq n$ and the fiber-motives at all $k < n$. Explicitly, 
\[
(\mathbb{L}^n-1) \frac{[\mathbb{M}^W_n(0)]-[\mathbb{M}^W_n(1)]}{[GL_n]} \]
is equal to
\[
 ([\mathbf{BS}^W_n(0)]-[\mathbf{BS}_n^W(1)])+\sum_{k=1}^{n-1}\frac{\mathbb{L}^{2k(n-k)}}{[GL_{n-k}]}(\mathbf{BS}^W_k(0)]-[\mathbf{BS}^W_k(1)])([\mathbb{M}^W_k(0)]-[\mathbb{M}^W_k(1)]) \]

We will next compute the first two terms in $U_W(t)$ and for $n=2$ the previous formula reduces to
\[
(\mathbb{L}^2-1)\frac{[\mathbb{M}^W_2(0)]-[\mathbb{M}^W_2(1)]}{[GL_2]} = [\mathbf{BS}^W_2(0)]-[\mathbf{BS}^W_2(1)]+ \frac{\mathbb{L}^2}{(\mathbb{L}-1)}([\mathbb{M}^W_1(0)]-[\mathbb{M}^W_1(1)])^2 
\]
and we have already that
\[
[ \mathbb{M}^W_1(1)]=[X_{DT}=1] \quad \text{and} \quad [\mathbb{M}^W_1(0)] = [X_{DT}=0] = (\mathbb{L}-1)[E_{DT}]+1 \]
so it remains to compute the difference of the Brauer-Severi motives $[\mathbf{BS}^W_2(0)]-[\mathbf{BS}^W_2(1)]$. 

From \cite{ReinekeBS} we deduce that $\wis{BS}_2(\mathbb{T}_{3,2})$ has a cellular decomposition as $\mathbb{A}^{10} \sqcup \mathbb{A}^8 \sqcup \mathbb{A}^8$ where the three cells have representatives
\[
\begin{cases}
\wis{cell}_1~:~v = \begin{bmatrix} 1 \\ 0 \end{bmatrix}, \quad X = \begin{bmatrix} 0 & p \\ 1 & r \end{bmatrix}, \quad 
Y = \begin{bmatrix} s & t \\ u & v \end{bmatrix}, \quad
Z = \begin{bmatrix} w & x \\ y & z \end{bmatrix}  \\
\\
\wis{cell}_2~:~v = \begin{bmatrix} 1 \\ 0 \end{bmatrix}, \quad X = \begin{bmatrix} n & p \\ 0 & r \end{bmatrix}, \quad 
Y = \begin{bmatrix} 0 & t \\ 1 & v \end{bmatrix}, \quad
Z = \begin{bmatrix} w & x \\ y & z \end{bmatrix} \\
\\
\wis{cell}_3~:~v = \begin{bmatrix} 1 \\ 0 \end{bmatrix}, \quad X = \begin{bmatrix} n & p \\ 0 & r \end{bmatrix}, \quad 
Y = \begin{bmatrix} s & t \\ 0 & v \end{bmatrix}, \quad
Z = \begin{bmatrix} 0 & x \\ 1 & z \end{bmatrix} 
\end{cases}
\]
It follows that $\wis{BS}_{3,2}^W(\lambda)$ decomposes as $\mathbf{S_1}(\lambda) \sqcup \mathbf{S_2}(\lambda) \sqcup \mathbf{S_3}(\lambda)$ where the subschemes $\mathbf{S_i}(\lambda)$ of $\mathbb{A}^{11-i}$ have defining equations
\[
\begin{cases}
\mathbf{S_1}(\lambda)~:~(C + Q_u.u + Q_y.y + Q_q)|_{n=0} = \lambda \\
\mathbf{S_2}(\lambda)~:~(C + Q_y.y + Q_u)|_{s=0} = \lambda \\
\mathbf{S_3}(\lambda)~:~(C + Q_y)|_{w=0} = \lambda
\end{cases}
\]
where
\[
\begin{cases}
C &=\frac{c}{3}(n^3+r^3+s^3+v^3+w^3+z^3)+(a+b)(rvz+nsw) \\
Q_q &= a(tz+sx)+b(vx+tw)+c p(r+n) \\
Q_u &= a(rx+pw)+b(pz+nx) + c t(v+s) \\
Q_y &= a(pv+nt) + b (rt+ps) + c x(z+w)
\end{cases}
\]
Note that in using the cellular decomposition, we set a variable equal to $1$. So, in order to retain a homogeneous form we let $\mathbb{G}_m$ act on $n,s,w,r,v,z$ with weight one, on $q,u,y$ with weight two and on $x,t,p$ with weight zero. Thus, we need a slight extension of \cite[Thm. 1.3]{Davison} as to allow $\mathbb{G}_m$ to act with weight two on certain variables.

\vskip 4mm

We will restrict to the case of a Sklyanin algebra with a point of order two, that is the case when $a=b$, which we may assume to be equal to $1$, and with $c \not= 0$.

\begin{lemma} With $a=b=1$ and $c \not= 0$ we have 
\[
\begin{cases}
[\mathbf{S_3}(0)] = \mathbb{L}^7+\mathbb{L}^5-\mathbb{L}^4 \\
[\mathbf{S_3}(1)] = \mathbb{L}^7 - \mathbb{L}^4
\end{cases}
\]
and therefore $[\mathbf{S_3}(0)]-[\mathbf{S_3}(1)] = \mathbb{L}^5$.
\end{lemma}

\begin{proof} The defining equation of $\mathbf{S_3}(\lambda)$ in $\mathbb{A}^8$ is
\[
\frac{c}{3}(n^3+r^3+s^3+v^3+z^3) + 2rvz + (v+s)p + (n+r)t + czx = \lambda \]

(1) : If $v+s \not= 0$ we can eliminate $p$ from the equation and get a contribution $\mathbb{L}^5(\mathbb{L}^2-\mathbb{L})$ as there are five free variables and $[v+s \not= 0]_{\mathbb{A}^2} = \mathbb{L}^2-\mathbb{L}$. Note that this is independent of the value of $\lambda$.

\vskip 2mm

(2) : If $v+s = 0$ we get the equation
\[
\frac{c}{3}(n^3+r^3+z^3) + 2rvz + (n+r)t + czx = \lambda \]
If we assume that in addition $n+r \not= 0$ we can eliminate $t$, then by an argument as above we obtain a contribution $\mathbb{L}^4(\mathbb{L}^2-\mathbb{L})$, again independent of the value of $\lambda$.

\vskip 2mm

(3) : If $v+s=0$ and $n+r=0$ we get as equation $\frac{c}{3}z^3 + 2rvz + czx = \lambda$. So, if $z \not= 0$ we can eliminate $x$ and get a term $\mathbb{L}^4(\mathbb{L}-1)$, independent of $\lambda$.

\vskip 2mm

(4) : If $v+s=0, n+r=0$ and $z=0$ we get the equation $0 = \lambda$. Hence, if $\lambda=1$ this gives no contribution, but if $\lambda=0$ we get a contribution $\mathbb{L}^5$.

\vskip 2mm

Summing up we get the claimed motives. 
\end{proof}

As we are only interested in the differences $[\mathbf{S_k}(0)]-[\mathbf{S_k}(1)]$ we will in the remaining computations only determine the difference of the motives in those subcases where the result can depend on the value of $\lambda$.

\begin{lemma} With $a=b=1$ and $c \not= 0$ we have 
\[
[\mathbf{S_2}(0)]-[\mathbf{S_2}(1)] = \mathbb{L}^6 + \mathbb{L}^3.[\mu_3].([X_0]-[X_1])
\]
where $X_{\lambda}$ is the locally closed subset in $\mathbb{A}^3$ (with variables $x,y,z$) defined by
\[
X_{\lambda} = \begin{cases}
x \not= 0 \\
x(3 \rho c z^2 - 3 \rho^2 c xz + 6 yz + (c^4+2c)x^2 - 3 \rho c^3 xy + 3 \rho^2 c^2 y^2) = 3 \lambda
\end{cases}
\]
and $\rho^3 = 1$.
\end{lemma}

\begin{proof} The defining equation of $\mathbf{S_2}(\lambda)$ in $\mathbb{A}^9$ is
\[
\frac{c}{3}(n^3+r^3+v^3+w^3+z^3)+2rvx+(vp+(n+r)t + c(z+w)x)y+ \]
\[
((r+n)x+(w+z)p+cvt) = \lambda \]

\vskip 2mm

(1) : If $vp+(n+r)t + c(z+w)x \not= 0$ we can eliminate $y$ from the equation, independent of the value of $\lambda$.

\vskip 2mm

(2) : If $vp+(n+r)t + c(z+w)x = 0$ and $v \not= 0$ we have 
\[
p = -\tfrac{n+r}{v}t-c \tfrac{z+w}{v}x \]
 and after substitution the equation becomes
\[
\frac{c}{3}(n^3+r^3+v^3+w^3+z^3)+2rvx+((r+n)-c\frac{(z+w)^2}{v})x + (cv-\frac{(n+r)(w+z)}{v})t = \lambda \]
If $v(r+n)-c(w+z)^2 \not= 0$ we can eliminate $x$ from the equation, and the remaining motive to consider, that is,
\[
[ vp+(n+r)t+c(z+w)x=0,v \not=0, v(r+n)-c(w+z)^2 \not= 0]_{\mathbb{A}^7} \]
does not depend on $\lambda$. 

If $v(r+n)-c(w+z)^2 = 0$ but $cv^2-(n+r)(w+z) \not= 0$ we can eliminate $t$, and again the resulting motive independent of $\lambda$, so does not contribute. 

\vskip 2mm

(3) : We arrive at the first subcase which depends on $\lambda$. The defining equations of the locally closed subset of $\mathbb{A}^5$ (we have eliminated $p$ and the variables $y,x$ and $t$ are free) are
\[
\begin{cases}
v \not= 0 \\
v(r+n)-c(w+z)^2 = 0 \\
cv^2-(n+r)(w+z) = 0 \\
\frac{c}{3}(n^3+r^3+v^3+w^3+z^3)+2rvz = \lambda
\end{cases} \]
From the first equation we obtain $r+n = \frac{c(w+z)^2}{v}$, and substituting this in the second equation gives 
\[
v^3=(w+z)^3 \qquad \text{ whence} \qquad \begin{cases} w = \rho v -z \\
 n = c \rho^2 v -r \end{cases} \]
for $\rho^3=1$, so we have three subcases to consider which are clearly isomorphic, giving a factor $[ \mu_3]$. 

If we substitute the obtained equations in the last equation, we obtain the locally closed subset in $\mathbb{A}^3$ (with remaining coefficients $r,v,z$)
\[
X_{\lambda} = \begin{cases}
v \not= 0 \\
v(3 \rho c z^2 - 3 \rho^2 c vz + 6 rz + (c^4+2c)v^2 - 3 \rho c^3 rv + 3 \rho^2 c^2 r^2) = 3 \lambda
\end{cases}
\]
Therefore, this subcase contributes a term equal to
\[
\mathbb{L}^3.[\mu_3].([X_0]-[X_1]) \]

\vskip 2mm

(4) : We have exhausted the $v \not= 0$ case, so from now on $v=0$ and we have to solve in $\mathbb{A}^7$
\[
\begin{cases} 
(n+r)t + c(z+w)x = 0 \\
\frac{c}{3}(n^3+r^3+w^3+z^3)+(r+n)x+(w+z)p = \lambda
\end{cases}
\]
If $w+z \not= 0$ we can eliminate $x$ from the first equation, substitute it in the second and eliminate $p$ from the second, all this independent of $\lambda$.

\vskip 2mm

(5) : If $w+z=0$ we have
\[
\begin{cases} (n+r)t = 0 \\ 
\frac{c}{3}(n^3+r^3)+(r+n)x = \lambda
\end{cases}
\]
So, if $r+n \not= 0$ we must have that $t=0$ and can eliminate $x$ from the second equation, independent of $\lambda$.

\vskip 2mm

(6) : The remaining case is when $y,x,t$ and $p$ are free variables and we have
\[
\begin{cases}
 v = 0 \\
  w+z =0 \\
  r+n = 0 
  \end{cases} \]
   and the remaining equation is $0 = \lambda$. So, for $\lambda = 1$ we get no contribution, whereas for $\lambda = 0$ we get a contribution $\mathbb{L}^6$.
\end{proof}

\begin{lemma} With $a=b=1$ and $c \not= 0$ we have 
\[
[\mathbf{S_1}(0)]-[\mathbf{S_1}(1)] = \mathbb{L}^7 + \mathbb{L}^3.[\mu_3].([X_0]-[X_1])
\]
where $X_{\lambda}$ is the locally closed subset in $\mathbb{A}^3$ (with variables $x,y,z$) defined by
\[
X_{\lambda} = \begin{cases}
x \not= 0 \\
x(3 \rho c z^2 - 3 \rho^2 c xz + 6 yz + (c^4+2c)x^2 - 3 \rho c^3 xy + 3 \rho^2 c^2 y^2) = 3 \lambda
\end{cases}
\]
and $\rho^3 = 1$.
\end{lemma}

\begin{proof}
The defining equation of $\mathbf{S_1}(\lambda)$ in $\mathbb{A}^{10}$ is equal to
\[
\frac{c}{3}(r^3+s^3+v^3+w^3+z^3)+2 rvz + ((w+z)p+c(v+s)t+r x)u + \]
\[
 ((v+s)p+r t + c(z+w)x)y + (crp + (z+w)t + (s+v)x) = \lambda \]
 Again, we will split the computations is subcases and only work out those for which the difference of motives may depend on $\lambda$.
 
 \vskip 3mm
 
 (1) : If $(w+z)p+c(v+s)t+r x \not= 0$ we can eliminate $u$ from the equation, independent of the value of $\lambda$.
 
 \vskip 2mm
 
 (2) : If $(w+z)p+c(v+s)t+r x = 0$, $u$ is a free variable and the equation becomes
 \[
 \frac{c}{3}(r^3+s^3+v^3+w^3+z^3)+2 rvz+((v+s)p+r t + c(z+w)x)y + (crp + (z+w)t + (s+v)x) = \lambda \]
  If $ry + (z+w) \not= 0$ we can eliminate $t$ from the equation, independent of $\lambda$.
  
  \vskip 2mm
  
  (3) : If $(w+z)p+c(v+s)t+r x = 0$ and $ry + (z+w) = 0$ and $r \not= 0$, then we have the equations
  \[
  \begin{cases}
  y = - \frac{z+w}{r} \\
  x=- \frac{w+z}{r}p-\frac{c(v+s)}{r}t
  \end{cases}
  \]
   and substitution gives us the equation
  \[
   \frac{c}{3}(r^3+s^3+v^3+w^3+z^3)+2 rvz - \frac{z+w}{r}((v+s)p + c(z+w)(- \frac{w+z}{r}p-\frac{c(v+s)}{r}t) + \]
   \[
   (crp  + (s+v)(- \frac{w+z}{r}p-\frac{c(v+s)}{r}t)) = \lambda \]
The coefficient of $t$ is equal to $-\frac{c(v+s)^2}{r}+\frac{z+w}{r}\frac{c^2(z+w)(v+s)}{r}$. Hence, if $c(z+w)^2(v+s)-r(v+s)^2 \not= 0$ we can eliminate $t$ from the equation, independent of $\lambda$.

\vskip 2mm

(4) : If $r \not= 0$, $(w+z)p+c(v+s)t+r x = 0$ and $ry + (z+w) = 0$ and $c(z+w)^2(v+s)-r(v+s)^2 = 0$, the equation becomes
\[
 \frac{c}{3}(r^3+s^3+v^3+w^3+z^3)+2 rvz +(\frac{c(z+w)^3}{r^2}-2\frac{(z+w)(s+v)}{r}+cr)p = \lambda \]
That is, if $c(z+w)^3-2r(z+w)(s+v)+cr^3 \not= 0$ we can eliminate $p$, independent of $\lambda$.

\vskip 2mm

(5) : The first subcase dependent on $\lambda$ is now that $u,p$ and $t$ are free variables and we have the following locally closed subset of $\mathbb{A}^5$ (in the remaining variables $r,s,v,w,z$)
\[
\begin{cases}
r \not= 0 \\
c(z+w)^2(v+s)-r(v+s)^2 = 0 \\
c(z+w)^3-2r(z+w)(s+v)+cr^3 = 0 \\
  \frac{c}{3}(r^3+s^3+v^3+w^3+z^3)+2 rvz = \lambda
  \end{cases}
  \]
  If $v+s \not= 0$ we have $r(v+s)=c(z+w)^2$ and substituting in the third equation gives $r^3=(z+w)^3$ whence $z+w = \rho r$ for $\rho^3=1$, but then also $c \rho^2 r = v+s$. If we substitute 
  \[
  \begin{cases} w = \rho r -z \\
  s = c \rho^2 -v
  \end{cases}
  \]
  in the last equation, we get the locally closed subset in $\mathbb{A}^3$, isomorphic to $X_{\lambda}$ of the previous case (interchanging the variables $r$ and $v$)
  \[
 X_{\lambda} = \begin{cases}
  r \not= 0 \\
  r(3 \rho c z^2+ 6 vz -3 \rho^2 c rz + 3 \rho^2 c^2 v^2-3 \rho c^3 rv + (c^4+2c) r^2) = \lambda
  \end{cases}
  \]
  Therefore, this subcase contributes a term equal to
\[
\mathbb{L}^3.[\mu_3].([X_0]-[X_1]) \]

\vskip 2mm  
  
(6) : From now on we may assume that $r=0$, together with $(w+z)p+c(v+s)t+rx=0$ and $ry+(z+w)=0$. But then, $z+w=0$ and the conditions are equivalent to the following system of equations in $\mathbb{A}^6$ (in the variables $s,t,v,p,x,y$).  Observe that we have $u$ and $w$ as extra free variables
\[
\begin{cases}
c(s+v)t = 0 \\
\frac{c}{3}(s^3+v^3)+(s+v)py + (s+v)x = \lambda
\end{cases}
\]
If $s+v \not= 0$ we have $t=0$ and can eliminate $x$ from the last equation, independent of $\lambda$.

\vskip 2mm

(7) : If $s+v=0$ we have $u,w,t,p,y,x,s$ as free variables and the remaining condition is $0 = \lambda$. That is, if $\lambda=1$ there is no contribution and for $\lambda=0$ we get a term $\mathbb{L}^7$.
\end{proof}

Summing up the three contributions, we have:

\begin{lemma} For the Brauer-Severi motives we have
\[
[\mathbf{BS}^W_2(0)]-[\mathbf{BS}^W_2(1)] = \mathbb{L}^7+\mathbb{L}^6+\mathbb{L}^5+2 \mathbb{L}^3 [ \mu_3] ([X_0]-[X_1]) \]
Therefore, the coefficient of $t^2$ in the series $U_W(t)$ is equal to
\[
\mathbb{L}^{-4} \frac{[\mathbb{M}^W_2(0)]-[\mathbb{M}^W_2(1)]}{[GL_2]} = 
\frac{\mathbb{L}(\mathbb{L}^3-1) + 2 [\mu_3]([X_0]-[X_1]) \mathbb{L}^{-1}(\mathbb{L}-1) + \mathbb{L}^{-2}N_1^2}{(\mathbb{L}^2-1)(\mathbb{L}-1)} \] 
\end{lemma}

\vskip 4mm

Remains to compute the motives $[X_{\lambda}]$ where
\[
X_{\lambda} = \begin{cases} x \not= 0 \\
x.(\rho c z^2- \rho^2 c x z + 2 y z + \frac{c^4+2 c}{3}x^2 - \rho c^3 x y + \rho^2 c^2 y^2)=\lambda
\end{cases}
\]
After performing the linear change of variables
\[
\begin{cases}
X = \sqrt{\frac{c^4+8c}{12}} x + i \frac{\sqrt{(c^3-1) \rho}}{c} z \\
Y = -\frac{c^2}{2} x + \rho c y + \frac{\rho^2}{c} z \\
Z = \sqrt{\frac{c^4+8c}{12}} x - i \frac{\sqrt{(c^3-1) \rho}}{c} z
\end{cases}
\]
we can express
\[
X_{\lambda} = \begin{cases}
X+Z \not= 0 \\
(X+Z)(Y^2+XZ) = \lambda
\end{cases} \]

\begin{lemma} With notations as above we have
\[
[ X_0 ] = (\mathbb{L}-1)^2 \quad \text{and} \quad [X_1] = (\mathbb{L}-1)^2 + [\mu_3] \mathbb{L} \]
\end{lemma}

\begin{proof} We have $[X_0]=[Y^2+XZ=0]_{\mathbb{A}^3} - [Y^2+XZ=0,X+Z=0]_{\mathbb{A}^3}$ which equals
\[
[Y^2+XZ=0]_{\mathbb{A}^3} - [(X+Y)(X-Y)=0]_{\mathbb{A}^2} = \mathbb{L}^2 - (2 \mathbb{L} -1) \]
As for $X_1$, we have for every $X+Z=a \not= 0$ 
\[
[Y^2 -X^2 + aX = \frac{1}{a}]_{\mathbb{A}^2} = \begin{cases} \mathbb{L}-1~\text{if $a^3 \not= 4$} \\ 2 \mathbb{L} -1~\text{if $a^3=4$} \end{cases} \]
as this is the affine part of a quadric $Y^2-X^2+aXU-\frac{1}{a}U^2=0$ in $\mathbb{P}^2$, having two points at infinity $U=0$, for every $a \not= 0$. The quadric has a unique singular point $[\tfrac{a}{2}:0:1]$ if and only if $a^3=4$.Therefore, 
\[
[X_1]=(\mathbb{L}-1-[\mu_3])(\mathbb{L}-1) + [\mu_3](2 \mathbb{L}-1). \]
\end{proof}

\begin{theorem} For the quaternionic Sklyanin algebra $S_{1,1,c}$ we have that the coefficient of the second term in the motivic Donaldson-Thomas series $U_W(t)$ is equal to
\[
\mathbb{L}^{-4} \frac{[\mathbb{M}^W_2(0)]-[\mathbb{M}^W_2(1)]}{[GL_2]} = 
\frac{\mathbb{L}(\mathbb{L}^3-1)  -2 [\mu_3]^2 (\mathbb{L}-1) + \mathbb{L}^{-2}N_1^2}{(\mathbb{L}^2-1)(\mathbb{L}-1)} \] 
\end{theorem}


\begin{thebibliography}{10}

\bibitem{ATV1} Mike Artin, John Tate and Michel Van den Bergh, {\it Modules over regular algebras of dimension $3$}, Invent. Math. 106 (1991) 335-388

\bibitem{ATV2} Mike Artin, John Tate and Michel Van den Bergh, {\it Some algebras associated to automorphisms of elliptic curves}, The Grothendieck Festschrift, Springer, 2007, 33-85

\bibitem{Bass} Hyman Bass, {\it Clifford algebras and spinor norms over a commutative ring}, Amer. J. Math. {\bf 96}, 156-206 (1974)

\bibitem{Bryan} J. Bryan and A. Morrison, {\it Motivic classes of commuting varieties via power structures}, J. Algebraic Geom. {\bf 24} (2015) 183-199


\bibitem{Cazz} Alberto Cazzaniga, Andrew Morrison, Brent Pym and Balazs Szendroi, {\it Motivic Donaldson-Thomas invariants for some quantized threefolds}, {\tt arXiv:1510.08116} (2015)

\bibitem{Davison}
Ben Davison and Sven Meinhardt, {\it Motivic DT-invariants for the one loop quiver with potential}, {\tt arXiv:1108.5956} (2011)


\bibitem{Kevin}
Kevin De Laet and Lieven Le Bruyn, {\it The geometry of representations of $3$-dimensional Sklyanin algebras}, Algebras and Representation Theory, {\bf 18} , 761-776 (2015)

\bibitem{Denef}
Jan Denef and Francois Loeser, {\it Geometry on arc spaces of algebraic varieties}, Europen Congress of Mathematics, Vol I (Barcelona, 2000), Progr. Math. {\bf 201}, Birkha\"ser (2001) 327-348

\bibitem{Lam}  T.Y. Lam, {\it The Algebraic Theory of Quadratic Forms}, Benjamin (1973)


\bibitem{LBVdB}
Lieven Le Bruyn and Michel Van den Bergh, {\it An explicit description of $\mathbb{T}(3,2)$}. In "Ring Theory, Proceedings Antwerp 1985",  109-113, Lecture Notes in Mathematics 1197, (1986).

\bibitem{LBsymbol}
Lieven Le Bruyn, {\it Sklyanin algebras and their symbols}, K-theory {\bf 8} (1994) 3-17

\bibitem{LBDT}
Lieven Le Bruyn, {\it Brauer-Severi motives and Donaldson-Thomas invariants of quantized threefolds}, {\tt arXiv:1604.08556} (2016)

\bibitem{Morrison}
Andrew Morrison, {\it Motivic invariants of quivers via dimensional reduction}, {\tt arXiv:1103.3819} (2011)




\bibitem{ReichWalton} Daniel J. Reich and Chelsea Walton, {\it Explicit representations of $3$-dimensional Sklyanin algebras associated to a point of order $2$}, {\tt arXiv:1512.09167}

\bibitem{ReinekeBS}
Markus Reineke, {\it Cohomology of non-commutative Hilbert schemes}, Alg. Repr. Theory {\bf 8} (2005) 541-561



\bibitem{SmithTate} S. Paul Smith and John Tate, {\it The centre of the $3$-dimensional and $4$-dimensional Sklyanin algebras}, K-theory, {\bf 8},19-63 (1994)

\bibitem{VdBBS}
Michel Van den Bergh, {\it The Brauer-Severi scheme of the trace ring of generic matrices}, Perspectives in Ring Theory (Antwerp 1987), NATO Adv. Sci. Inst. Ser. C Math. Phys. Sci., vol 233, Kluwer (1988)


\bibitem{Walton}
Chelsea Walton, {\it Representation theory of three-dimensional Sklyanin algebras}, Nuclear Phys. B, {\bf 860}, 167-185 (2012)

\bibitem{Walton2}
Chelsea Walton, Xingting Wang and Milen Yakimov, {\it The Poisson geometry of the 3-dimensional Sklyanin algebras}, {\tt arXiv:1704.04975} (2017)

\end{thebibliography}
\end{document}